\documentclass[10pt,onesite,draft]{article}
\newfam\cyrfam

\usepackage{amssymb,latexsym,amscd,amsmath,amsfonts, amsthm,enumerate}

\newtheorem{theorem}{Theorem}[section]

\newtheorem{proposition}[theorem]{Proposition}
\newtheorem{corollary}[theorem]{Corollary}

\newtheoremstyle{definition}
  {6pt}
  {6pt}
  {}
  {}
  {\bfseries}
  {.}
  {.5em}
  {}%

\theoremstyle{definition}

\newtheoremstyle{remark}
  {6pt}
  {6pt}
  {}
  {}
  {\bfseries}
  {.}
  {.5em}
  {}%

\theoremstyle{remark}

\newtheoremstyle{note}
  {6pt}
  {6pt}
  {}
  {}
  {\bfseries}
  {.}
  {.5em}
  {}%

\theoremstyle{note}
\newtheorem{note}[theorem]{Note}

\makeatletter
\renewcommand\@makefntext[1]{%
\setlength\parindent{1em}%
\noindent \makebox[1.8em][r]{}{#1}} \makeatother

\begin{document}
\parskip 4pt
\large \setlength{\baselineskip}{15 truept}
\setlength{\oddsidemargin} {0.5in} \overfullrule=0mm
\date{}
\title{\bf \large  MIXED MULTIPLICITIES AND  THE MULTIPLICITY \\OF REES MODULES OF REDUCTIONS  \\(Published in J. Algebra Appl.)
}
\def\b{\vntime}
\author{
Truong Thi Hong Thanh and  Duong Quoc Viet\\
\small Department of Mathematics, Hanoi National University of Education\\
\small 136 Xuan Thuy Street, Hanoi, Vietnam\\
\small Email: thanhtth@gmail.com, vduong99@gmail.com}
 \date{}
\maketitle \centerline{\parbox[c]{10.1cm}{ \small{\bf ABSTRACT:} This paper shows that mixed
multiplicities and the multiplicity of Rees modules of good
filtrations and that of their reductions are the same. As an application of this result, we obtain interesting results on
  mixed multiplicities
 and the multiplicity of Rees modules.}}
 \section{Introduction} \noindent
\noindent In past years, mixed multiplicities and the multiplicity
of Rees modules have attracted much attention (see e.g. 2, 6$-$9, 13$-$31). However
up to now, results on the relationship between mixed multiplicities (multiplicities of Rees modules) and mixed multiplicities (multiplicities of Rees modules) of reductions, even in the case of ideals,  are still restricted.
 In this paper, wishing to have broader applications, we consider this problem in the context of good filtrations.

\enlargethispage{1cm} \footnotetext{\begin{itemize}\item[ ]{\bf Mathematics Subject  Classification (2010):} Primary 13H15.
Secondary 13C15, 13D40, 14C17.  \item[ ]{\bf  Key words and
phrases:} Multiplicity, mixed multiplicity, good filtration, Rees
ring.
\end{itemize}}

Let $(A, \frak m)$ be a Noetherian local
ring with maximal ideal $\frak m$; $M$ be a finitely generated
$A$-module.
 A {\it filtration} of ideals
$F=\{(F)_n\}_{n\geqslant0}$ in $A$ is a decreasing chain of ideals
$A=(F)_0\varsupsetneq (F)_1\supseteq\cdots\supseteq (F)_n\cdots$
such that $(F)_{m}(F)_{n}\subseteq (F)_{m+n}$ for all
$m,n\geqslant0.$ For any ideal $I$ of $A,$ the filtration
$\{I^n\}_{n\geqslant0}$ determined by the powers of $I$ is called
the {\it $I$-adic filtration}. Let $I_1, \ldots, I_s$ be ideals
and $ F_1 = \{(F_1)_n\}_{n\geqslant 0},\ldots,$ $F_s =
\{(F_s)_n\}_{n\geqslant0}$ be filtrations in $A$. Put
\begin{align*}&{\bf n} =
(n_1,\ldots,n_s);{\bf k} = (k_1,\ldots,k_s); {\bf 0}=(0,\ldots,0);
{\bf 1}=(1,\ldots,1);\\ &{\bf n^k}= n_1^{k_1}\cdots n_s^{k_s};
\mathbf{k}!= k_1!\cdots k_s!; \;|{\bf k}| = k_1 + \cdots + k_s;\\
 &{\bf I} = I_1, \ldots, I_s; {\bf I}^{[\bf k]} = I_1^{[k_1]}, \ldots, I_s^{[k_s]};\;
 \mathbb{I}^{\bf n} = I_1^{n_1} \cdots I_s^{n_s};\\
&\mathbf{F} = F_1, \ldots, F_s;\;{\bf F}^{[\bf k]} = F_1^{[k_1]},
\ldots, F_s^{[k_s]};   \; \mathbb{F}_{\bf n} = (F_1)_{n_1} \cdots
(F_s)_{n_s}. \end{align*} Set $T^{\bf n}= t_1^{n_1}\cdots
t_s^{n_s},$ here $t_i$ is a variable over $A$ for all
$i=1,\ldots,s.$ Denote by $$\frak R(\mathrm{\bf I}; A) =
  \bigoplus_{\bf n\geqslant 0}\mathbb{ I}^{\mathrm{\bf n}}T^{\bf n}\; \text{ and }\; \frak R(\mathrm{\bf I}; M) =
   \bigoplus_{\bf n \geqslant 0}\mathbb{ I}^{\mathrm{\bf n}}MT^{\bf n} $$
  the multi-Rees algebra and  the multi-Rees module of the ideals $\mathbf{I}$ with respect to $M$, respectively; and by
  $\frak R(\mathrm{\bf F}; A) =
  \bigoplus_{\bf n\geqslant 0}\mathbb{F}_{\mathrm{\bf n}}T^{\bf n} \; \text{ and }\; \frak R(\mathrm{\bf F}; M) =
  \bigoplus_{\bf n\geqslant 0}\mathbb{F}_{\mathrm{\bf n}}MT^{\bf n}$
   the multi-Rees algebra and  the multi-Rees module with respect to $M$ of the filtrations $\mathbf{F}$, respectively.
   Set $\frak R(\mathrm{\bf I}; A)_+ = \bigoplus_{|{\bf n}|> 0}\mathbb{ I}^{\mathrm{\bf n}}T^{\bf n}$ and  $\frak R(\mathrm{\bf F}; A)_+ = \bigoplus_{|{\bf n}|> 0}\mathbb{F}_{\mathrm{\bf n}}T^{\bf n}$.

A filtration $ F$ is called an {\it $I$-good filtration} if
$I(F)_n\subseteq (F)_{n+1}$ for all $n\geqslant0$ and
$(F)_{n+1}=I(F)_n$ for all large $n$. In this case, $I$ is  also
called a {\it reduction} of $ F$. The filtration $F$ is  called a
{\it good filtration} if it is an $I$-good filtration for some
ideal $I.$
 A
 filtration $F$ is called {\it $\mathfrak{m}$-primary} if $(F)_1$ is an
$\mathfrak{m}$-primary ideal.  Let $F$ be an
$\mathfrak{m}$-primary good filtration and $\mathbf{F}$ be good
filtrations of ideals in $A$ such that $I = (F_1)_1 \cdots
(F_s)_1$ is not contained in $\sqrt{\mathrm{Ann}M}$. Set $q=\dim
M/0_M:I^\infty$. Then the paper shows  that  $
\ell_A\Big(\frac{(F)_{n_0}\mathbb{F}_{\bf
n}M}{(F)_{n_0+1}\mathbb{F}_{\bf n}M}\Big)$
 is a polynomial of total degree $q-1$ in $n_0, \bf n$ for all large  $n_0, \bf n$ (see Proposition \ref{pr1}).
 Now if we  write the terms of total degree $q-1$ in this polynomial in the form
$\sum_{k_0 + |\mathbf{k}|=q-1}e(F^{[k_0 +
1]},\mathbf{F}^{[\mathbf{k}]}; M)
\frac{n_0^{k_0}\mathbf{n}^{\mathbf{k}}}{k_0!\mathbf{k}!},$ then
$e(F^{[k_0 + 1]},\mathbf{F}^{[\mathbf{k}]}; M)$ are non-negative
integers not all zero. And $e(F^{[k_0 +
1]},\mathbf{F}^{[\mathbf{k}]}; M)$ is called
 the {\it mixed multiplicity} of $M$ with respect to the good filtrations $ F, \mathbf{F}$
 of the type $(k_0 + 1,\mathbf{k})$.
Remember that in the case where  $F$ is a $J$-adic filtration and
$ F_i$ is an $I_i$-adic filtration for all $i = 1, \ldots, s$,
$e(F^{[k_0 + 1]},\mathbf{F}^{[\mathbf{k}]}; M)$ is denoted by
$e(J^{[k_0 + 1]},\mathbf{I}^{[\mathbf{k}]}; M)$ and called the
{\it mixed multiplicity} of $M$ with respect to the ideals $J,
\mathbf{I}$ of the type $(k_0 + 1,\mathbf{k})$ (see e.g. \cite{MV,
Vi1, VD}). And  as one might expect, we obtain the following result
which is also the main theorem of this paper.

\begin{theorem}[Theorem \ref{thm2.3}] Let $ F$ be an $\mathfrak{m}$-primary good
filtration and  $ \mathbf{F}$ be good filtrations of ideals in $A$
such that $I= (F_1)_1 \cdots (F_s)_1$ is not contained in
$\sqrt{\mathrm{Ann}M}$. Let $J,  I_1, \ldots, I_s$ be reductions of
$ F,  F_1, \ldots,  F_s$, respectively. Let $\frak I$ be an
$\mathfrak{m}$-primary ideal of $A$. Then the following statements
hold.
\begin{enumerate}[{\rm (i)}]
\item $e(F^{[k_0 + 1]},\mathbf{F}^{[\mathbf{k}]}; M) = e(J^{[k_0 +
1]},\mathbf{I}^{[\mathbf{k}]};  M)$. \item  Assume  that
$\mathrm{ht}\dfrac{I + \mathrm{Ann}M}{\mathrm{Ann}M}> 0.$ Then we
have
$$e((\frak I, \frak R(\mathrm{\bf F}; A)_+);\frak R(\mathbf{F};
M)) = e((\frak I, \frak R(\mathrm{\bf I}; A)_+); \frak
R(\mathbf{I}; M)).$$
\end{enumerate}
\end{theorem}
Now, if
$\overline{\frak R(I; A)}$ denotes  the integral closure of the
Rees algebra $\frak R(I; A) = A[It]$ in the polynomial ring
$A[t],$ then applying Theorem 1.1, we obtain the following.
\begin{corollary}[Corollary \ref{co2.7}] Let $A$ be an analytically unramified ring and $I$ an
ideal of positive height of $A.$  Then $e(\overline{\frak R(I;
A)}) =e(\frak R(I; A)).$
\end{corollary}

Moreover, using Theorem 1.1, we easily
obtain  results for multiplicities of good filtrations: expressing
mixed multiplicities of modules in terms of the multiplicity of
 joint reductions (Corollary \ref{co2.8}); the multiplicity  of Rees
modules as the sum of the mixed multiplicities (Corollary
\ref{co2.4}); the additivity on exact sequences (Corollary
\ref{co2.5}) and the additivity and reduction formula (Corollary
\ref{co2.6}) for mixed multiplicities and the multiplicity of Rees
modules; the relationship between mixed multiplicities of modules
   and mixed multiplicities of rings
via the rank of modules (Corollary \ref{co3.b}).

This paper is divided into three  sections. In Section 2, first we
define
 mixed multiplicities of good filtrations (see Proposition
\ref{pr1}); next we prove the main theorem (Theorem \ref{thm2.3}) and
give Corollary \ref{co2.7} and Corollary \ref{co2.8}.
  Section 3 discusses  some applications of Theorem \ref{thm2.3}
  to
  formulas for  mixed multiplicities and multiplicities of Rees
  modules (see Corol. \ref{co2.4}, Corol. \ref{co2.5}, Corol.
\ref{co2.6}, Corol. \ref{co3.b}).

 \section{Mixed multiplicities and multiplicities of Rees modules}
In this section, first we give Proposition \ref{pr1} which is used
for defining mixed multiplicities of good filtrations, and prove
the main theorem (Theorem \ref{thm2.3}) and give some applications
of this theorem to  multiplicities of good filtrations.

 $(A, \frak m)$ denotes a Noetherian local ring
with maximal ideal $\frak m.$ Recall that a {\it filtration} of
ideals $F=\{(F)_n\}_{n\geqslant0}$ in $A$ is a decreasing chain of
ideals $$A=(F)_0\varsupsetneq (F)_1\supseteq\cdots\supseteq
(F)_n\cdots$$ such that $(F)_{m}(F)_{n}\subseteq (F)_{m+n}$ for
all $m,n\geqslant0.$ For any ideal $I$ of $A,$ the filtration
$\{I^n\}_{n\geqslant0}$ determined by the powers of $I$ is called
the {\it $I$-adic filtration}. Let $I$ be an ideal of $A$ and let
$ F=\{(F)_n\}_{n\geqslant0}$ be a filtration of ideals in $A$.
Then we say that $ F$ is an {\it $I$-good filtration} if
$I(F)_n\subseteq (F)_{n+1}$ for all $n\geqslant0$ and
$(F)_{n+1}=I(F)_n$ for all sufficiently large $n$. In this case,
$I\subseteq (F)_1$.  If $ F$  is  an $I$-good filtration, then $I$
is  called a {\it reduction} of $ F$.  The filtration $ F$ is
called a {\it good filtration} if it is an $I$-good filtration for
some ideal $I$ of $A.$  It is easily seen that $ F$ is a good
filtration if and only if $ F$ is an $(F)_1$-good filtration and
in this case, $(F)_1$ is a reduction of $ F$.  There are numerous
examples of non-ideal-adic good filtrations: for an ideal $I$ of
$A$ containing  a non-zero-divisor,
$\{\widetilde{I^n}\}_{n\geqslant 0}$ is a good filtration, where
$\widetilde{I^n}$ is the Ratliff-Rush closure of $I^n$ (see
\cite[Theorem 2.1]{RR}). Furthermore, if $A$ is an analytically
unramified ring, then $\{\overline{I^n}\}_{n\geqslant 0}$ is an
$I$-good filtration (see \cite{R2}) and if $A$ is an analytically
unramified ring containing a field, then $\{(I^n)^*\}_{n\geqslant
0}$ is an $I$-good filtration, here $\overline{I^n}$ and $(I^n)^*$
denote the integral  closure and tight closure of $I^n$,
respectively (see \cite[Ch. 13]{HS}).

    A filtration $F=\{(F)_n\}_{n\geqslant 0}$ is {\it
$\mathfrak{m}$-primary} if $(F)_1$ is an $\mathfrak{m}$-primary
ideal. Note that $\mathfrak{m}$-primary filtrations are sometimes
called Hilbert filtrations (see e.g. \cite{HM}).

 Let
 $M$ be a finitely generated
$A$-module and $I_1, \ldots, I_s$ be ideals of $A.$  Let $$ F_1 =
\{(F_1)_n\}_{n\geqslant 0},\ldots, F_s =
\{(F_s)_n\}_{n\geqslant0}$$ be filtrations in $A$ $(s > 0)$. Set
\begin{align*}&{\bf n} =
(n_1,\ldots,n_s);{\bf k} = (k_1,\ldots,k_s); {\bf 0}=(0,\ldots,0);
 {\bf 1}=(1,\ldots,1) \in \mathbb{N}^s;\\ &{\bf n^k}= n_1^{k_1}\cdots
n_s^{k_s};
\mathbf{k}!= k_1!\cdots k_s!; \;|{\bf k}| = k_1 + \cdots + k_s;\\
 &{\bf I} = I_1, \ldots, I_s; {\bf I}^{[\bf k]} = I_1^{[k_1]}, \ldots, I_s^{[k_s]};\;
 \mathbb{I}^{\bf n} = I_1^{n_1} \cdots I_s^{n_s};\\
&\mathbf{F} = F_1, \ldots, F_s;\;{\bf F}^{[\bf k]} = F_1^{[k_1]},
\ldots, F_s^{[k_s]};   \; \mathbb{F}_{\bf n} = (F_1)_{n_1} \cdots
(F_s)_{n_s}. \end{align*} Put $T^{\bf n}= t_1^{n_1}\cdots
t_s^{n_s},$ here $t_i$ is a variable over $A$ for all
$i=1,\ldots,s.$ Denote by $$\frak R(\mathrm{\bf I}; A) =
  \bigoplus_{\bf n\geqslant 0}\mathbb{ I}^{\mathrm{\bf n}}T^{\bf n}\; \text{ and }\; \frak R(\mathrm{\bf I}; M) =
   \bigoplus_{\bf n \geqslant 0}\mathbb{ I}^{\mathrm{\bf n}}MT^{\bf n} $$
  the multi-Rees algebra and  the multi-Rees module with respect to $M$ of the ideals $\mathbf{I}$, respectively; and by
  $$\frak R(\mathrm{\bf F}; A) =
  \bigoplus_{\bf n\geqslant 0}\mathbb{F}_{\mathrm{\bf n}}T^{\bf n} \; \text{ and }\; \frak R(\mathrm{\bf F}; M) =
  \bigoplus_{\bf n\geqslant 0}\mathbb{F}_{\mathrm{\bf n}}MT^{\bf n}$$
   the multi-Rees algebra and  the multi-Rees module of the filtrations $\mathbf{F}$  with respect to $M$, respectively.
   Set $\frak R(\mathrm{\bf I}; A)_+ = \bigoplus_{|{\bf n}|> 0}\mathbb{ I}^{\mathrm{\bf n}}T^{\bf n}$ and  $\frak R(\mathrm{\bf F}; A)_+ = \bigoplus_{|{\bf n}|> 0}\mathbb{F}_{\mathrm{\bf n}}T^{\bf n}$.

 Let $\mathbf F:   F_1 =
\{(F_1)_n\}_{n\geqslant0},\ldots, F_s = \{(F_s)_n\}_{n\geqslant0}$
 be good filtrations of ideals
 in $A$ such that $I = (F_1)_1 \cdots (F_s)_1$ is not contained in $\sqrt{\mathrm{Ann}M}$ and
 $F =  \{(F)_n\}_{n\geqslant0}$ be an $\mathfrak{m}$-primary good filtration.
 Now, we consider the $\mathbb{N}^{s+1}$-graded
algebra: $$\mathcal S= \bigoplus_{n_0\geqslant 0,{\bf n}\geqslant
{\bf 0}}\frac{(F)_{n_0}\mathbb{F}_{\bf
n}}{(F)_{n_0+1}\mathbb{F}_{\bf n}}$$
 and the $\mathbb{N}^{s+1}$-graded $\mathcal S$-module:
 $ \mathcal M = \bigoplus_{n_0\geqslant 0,{\bf n}\geqslant {\bf 0}}
 \frac{(F)_{n_0}\mathbb{F}_{\bf n}M}{(F)_{n_0+1}\mathbb{F}_{\bf
 n}M}.$ Then we get
 the Bhattacharya function of $ \mathcal M $ (see \cite{Bh})
$$B(n_0,{\bf n}; F, \mathbf F; M ) = \ell_A\Big(\frac{(F)_{n_0}\mathbb{F}_{\bf
n}M}{(F)_{n_0+1}\mathbb{F}_{\bf n}M}\Big).$$ Note that $\mathcal
S$ is not a standard graded algebra. This is an obstruction for
proving $B(n_0,{\bf n}; F, \mathbf F; M )$ is a polynomial for all
large $n_0, \bf n$. So we need the following note which plays an
important role in the approach of this  paper.

\begin{note}\label{no2.1} For $n \geqslant 1$, set $\mathcal{I}_n = (F)_n(F_1)_n\cdots (F_s)_n$.
Assume that $J,  I_1, \ldots, I_s$ are reductions of $ F,  F_1,
\ldots,  F_s$, respectively. Then there exists a large enough
positive integer $c$ such that $(F)_{n_0 + c} = J^{n_0}(F)_c$,
$(F_i)_{n_i + c} = I_i^{n_i}(F_i)_c$ for  all  $n_0, n_i$ and all
$i = 1, \ldots, s$. Thus $(F)_{n_0 + c}\mathbb{F}_{{\bf n} + c{\bf
1}} = J^{n_0}\mathbb{I}^{\mathbf{n}}\mathcal{I}_c$ for all
$n_0,{\bf n}.$ So
$$\frac{(F)_{n_0 +
c}\mathbb{F}_{{\bf n} + c{\bf 1}}M}{(F)_{n_0+1 +
c}\mathbb{F}_{{\bf n} + c{\bf 1}}M}
=\dfrac{J^{n_0}\mathbb{I}^{\mathbf{n}}\mathcal{I}_cM}{J^{n_0 +
1}\mathbb{I}^{\mathbf{n}}\mathcal{I}_cM},$$  and hence
$$B(n_0+c,{\bf n} + c{\bf 1}; F, \mathbf F; M)= B(n_0,{\bf n}; J, \mathbf I; \mathcal{I}_cM )$$
  for all $n_0,{\bf n},$ here $B(n_0,{\bf n}; J, \mathbf I; \mathcal{I}_cM ) =
\ell_A\Big(\dfrac{J^{n_0}\mathbb{I}^{\mathbf{n}}\mathcal{I}_cM}{J^{n_0
+ 1}\mathbb{I}^{\mathbf{n}}\mathcal{I}_cM}\Big).$
\end{note}

 Using Note \ref{no2.1}, we prove the
following proposition.

\begin{proposition}\label{pr1}
Let $F$ be an $\mathfrak{m}$-primary good filtration and
$\mathbf{F}$ be good filtrations such that $I= (F_1)_1 \cdots
(F_s)_1$ is not contained in $\sqrt{\mathrm{Ann}M}$.  Set $q=\dim
M/0_M:I^\infty$. Then $B(n_0,{\bf n}; F, \mathbf F; M )$ is a
polynomial of degree $q-1$ in $n_0, \bf n$ for all large  $n_0,
\bf n$.
\end{proposition}

\begin{proof}
Let $J,  I_1, \ldots ,I_s$ be reductions of $ F,  F_1, \ldots,
F_s$, respectively. Then it is easily to see that $0_M:
(I_1\cdots I_s)^\infty = 0_M: I^\infty$. By Note \ref{no2.1},
there exists a positive integer $c$ such that
$$B(n_0+c,{\bf n} + c{\bf 1}; F, \mathbf F; M)= B(n_0,{\bf n}; J, \mathbf I; \mathcal{I}_cM )$$
  for all $n_0,{\bf n}.$
 Since
$B(n_0,{\bf n}; J, \mathbf I; \mathcal{I}_cM )$ is a polynomial in
$n_0, \bf n$ for all large  $n_0, \bf n$ by \cite[Theorem
4.1]{HHRT} and this polynomial has degree: $\dim
\mathcal{I}_cM/0_{\mathcal{I}_cM}: I^\infty - 1$ by
\cite[Proposition 3.1 (i)]{Vi1} (see \cite{MV}), it is easily seen
that $B(n_0,{\bf n}; F, \mathbf F; M)$ is a polynomial of degree
$\dim \mathcal{I}_cM/0_{\mathcal{I}_cM}: I^\infty - 1$ for all
large $n_0, \bf n.$ Note that
  $(0_M: I^\infty) : \mathcal{I}_c = 0_M: I^\infty$.
Therefore
\begin{align*}
\dim \mathcal{I}_cM/0_{\mathcal{I}_cM}: I^\infty &= \dim
A/(0_M:{I}^\infty): \mathcal{I}_cM =
\dim A/(0_M: {I}^\infty: \mathcal{I}_c):M\\
& = \dim A/(0_M: I^\infty):M = \dim M/0_M: I^\infty = q.
\end{align*}
Hence $B(n_0,{\bf n}; F, \mathbf F; M)$ is a polynomial of degree
$q-1$ for all large $n_0, \bf n$.
\end{proof}

With assumptions as in Proposition \ref{pr1}, $B(n_0,{\bf n}; F,
\mathbf F; M)$ is a polynomial of degree $q-1$ for all large $n_0,
\bf n.$ Denote by $P(n_0,{\bf n}; F, \mathbf F; M)$ this
polynomial.   Write the terms of total degree $q-1$ in the
polynomial $P(n_0,{\bf n}; F, \mathbf F; M)$  in the form
$$\sum_{k_0 + |\mathbf{k}|=q-1}e(F^{[k_0 +
1]},\mathbf{F}^{[\mathbf{k}]}; M)
\frac{n_0^{k_0}\mathbf{n}^{\mathbf{k}}}{k_0!\mathbf{k}!},$$ then
it is easily seen that
 $e(F^{[k_0 + 1]},\mathbf{F}^{[\mathbf{k}]}; M)$ are
non-negative integers not all zero. We call $e(F^{[k_0 +
1]},\mathbf{F}^{[\mathbf{k}]}; M)$
 the {\it mixed multiplicity} of $M$ with respect to the good filtrations
 $ F, \mathbf{F}$ of the type $(k_0 + 1,\mathbf{k})$.

In the case that  $ F_i$ is an $I_i$-adic filtration for all $i =
1, \ldots, s$ and $F$ is a $J$-adic filtration,    where $J$ is an
$\frak m$-primary ideal and $I_1\cdots I_s$ is not contained in
$\sqrt{\mathrm{Ann}M}$, then $e(F^{[k_0 +
1]},\mathbf{F}^{[\mathbf{k}]}; M)$ is denoted by $e(J^{[k_0 +
1]},\mathbf{I}^{[\mathbf{k}]}; M)$ and called the {\it mixed
multiplicity} of $M$ with respect to the ideals $J, \mathbf{I}$ of
the type $(k_0 + 1,\mathbf{k})$ (see e.g. \cite{MV, Vi1, VD}).

Then the main result of this paper is the following theorem.

\begin{theorem}\label{thm2.3} Let $ F$ be an $\mathfrak{m}$-primary good
filtration and  $ \mathbf{F}$ be good filtrations of ideals in $A$
such that $I= (F_1)_1 \cdots (F_s)_1$ is not contained in
$\sqrt{\mathrm{Ann}M}$. Let $J,  I_1, \ldots, I_s$ be reductions of
$ F,  F_1, \ldots,  F_s$, respectively. Let $\frak I$ be an
$\mathfrak{m}$-primary ideal of $A$. Then the following statements
hold.
\begin{enumerate}[{\rm (i)}]
\item $e(F^{[k_0 + 1]},\mathbf{F}^{[\mathbf{k}]}; M) = e(J^{[k_0 +
1]},\mathbf{I}^{[\mathbf{k}]};  M)$. \item  Assume  that
$\mathrm{ht}\dfrac{I + \mathrm{Ann}M}{\mathrm{Ann}M}> 0.$ Then we
have
$$e((\frak I, \frak R(\mathrm{\bf F}; A)_+);\frak R(\mathbf{F};
M)) = e((\frak I, \frak R(\mathrm{\bf I}; A)_+); \frak
R(\mathbf{I}; M)).$$
\end{enumerate}
\end{theorem}
\begin{proof}
The proof of (i): From Note \ref{no2.1},  there exists a positive
integer $c$ such that $$B(n_0+c,{\bf n} + c{\bf 1}; F, \mathbf F;
M)= B(n_0,{\bf n}; J, \mathbf I; \mathcal{I}_cM )$$
  for all $n_0,{\bf n},$ here
$\mathcal{I}_c = (F)_c(F_1)_c\cdots (F_s)_c$.  Hence by the
definition of the mixed multiplicities, we get  $$e(F^{[k_0 +
1]},\mathbf{F}^{[\mathbf{k}]}; M) = e(J^{[k_0 +
1]},\mathbf{I}^{[\mathbf{k}]};  \mathcal{I}_cM). \eqno(1)$$  Note
that $0_M: (I_1\cdots I_s)^\infty = 0_M: I^\infty$. Set
$$\overline{M} = M/0_M: I^\infty.$$  By \cite[Proposition 3.1
(ii)]{MV} (see \cite[Proposition 3.1 (ii)]{Vi1}), we have
$$e(J^{[k_0 + 1]},\mathbf{I}^{[\mathbf{k}]}; \mathcal{I}_cM)=
e(J^{[k_0 + 1]},\mathbf{I}^{[\mathbf{k}]};
\mathcal{I}_c\overline{M}). \eqno(2)$$  Consider the short exact
sequence of $A$-modules:
$$0\rightarrow \mathcal{I}_c\overline{M} \rightarrow \overline{M}
\rightarrow \overline{M}/\mathcal{I}_c\overline{M}\rightarrow 0.$$
Since  $\mathrm{Ann}\overline{M} : \mathcal{I}_c =
\mathrm{Ann}\overline{M}$, it follows that
$\mathrm{ht}\dfrac{\mathcal{I}_c +
\mathrm{Ann}\overline{M}}{\mathrm{Ann}\overline{M}} > 0$. Hence
$$\dim \overline{M}/\mathcal{I}_c\overline{M} < \dim \overline{M}.$$
 So by  the  additivity of mixed multiplicities \cite[Corollary 3.9 (ii)(a)]{VT1}, we get
$$e(J^{[k_0 + 1]},\mathbf{I}^{[\mathbf{k}]};  \mathcal{I}_c\overline{M})
= e(J^{[k_0 + 1]},\mathbf{I}^{[\mathbf{k}]};  \overline{M}).
\eqno(3)$$ Also by \cite[Proposition 3.1 (ii)]{MV} (see
\cite[Proposition 3.1 (ii)]{Vi1}, we obtain
$$e(J^{[k_0 + 1]},\mathbf{I}^{[\mathbf{k}]};  \overline{M}) = e(J^{[k_0 + 1]},\mathbf{I}^{[\mathbf{k}]};
 {M}). \eqno(4)$$
Therefore, by (1), (2), (3) and (4),  we have
$$e(F^{[k_0 + 1]},\mathbf{F}^{[\mathbf{k}]}; M) = e(J^{[k_0 + 1]},\mathbf{I}^{[\mathbf{k}]};  {M}).$$

The proof of (ii): Set $\frak U = (\frak I, \frak R(\mathrm{\bf
I}; A)_+)$ and $\frak V = (\frak I, \frak R(\mathrm{\bf F};
A)_+)$. At first as in Note \ref{no2.1}, there exists a positive
integer $c$ such that $\mathbb{F}_{{\bf n} + c{\bf 1}} =
\mathbb{I}^{\mathbf{n}}{I}_c$ for all ${\bf n}$,
 here $$I_c
= (F_1)_c \cdots (F_s)_c.$$ So  it can be verified that $I_c\frak
R(\mathbf{F}; M) = \frak R(\mathbf{I}; I_cM)$ and $\frak
V^nI_c\frak R(\mathbf{F}; M) = \frak U^n \frak R(\mathbf{I};
I_cM)$ for all $n \ge 1$. Hence $\dfrac{\frak R(\mathbf{I};
{I}_cM)}{\frak U^n \frak R(\mathbf{I}; {I}_cM)}=\dfrac{{I}_c\frak
R(\mathbf{F}; M) }{\frak V^n{I}_c\frak R(\mathbf{F}; M)}$ for all
$n \ge 1$.
 Now, for any $n \ge 1,$ assume that $\ell_{\frak R(\mathrm{\bf
F}; A)}\bigg(\dfrac{{I}_c\frak R(\mathbf{F}; M) }{\frak
V^n{I}_c\frak R(\mathbf{F}; M)}\bigg) = u.$ Then there exists a
composition series
 $$0 = {\cal M}_0
\subseteq {\cal M}_1 \subseteq {\cal M}_2 \subseteq\cdots
\subseteq {\cal M}_u=\dfrac{{I}_c\frak R(\mathbf{F}; M) }{\frak
V^n{I}_c\frak R(\mathbf{F}; M)} \eqno(5)$$ of the $\frak
R(\mathrm{\bf F}; A)$-module $\dfrac{{I}_c\frak R(\mathbf{F}; M)
}{\frak V^n{I}_c\frak R(\mathbf{F}; M)},$ i.e., ${\cal
M}_{i+1}/{\cal M}_i \cong \frak R(\mathbf{F}; A)/(\frak m, \frak
R(\mathrm{\bf F}; A)_+)$ for all $0 \le i \le u-1.$ Since $\frak
R(\mathbf{I}; A)/(\frak m, \frak R(\mathrm{\bf I}; A)_+) \cong
A/\frak m \cong \frak R(\mathbf{F}; A)/(\frak m, \frak
R(\mathrm{\bf F}; A)_+)$ and $\dfrac{\frak R(\mathbf{I};
{I}_cM)}{\frak U^n \frak R(\mathbf{I}; {I}_cM)}=\dfrac{{I}_c\frak
R(\mathbf{F}; M) }{\frak V^n{I}_c\frak R(\mathbf{F}; M)}$ is also
an $\frak R(\mathbf{I}; A)$-module, it follows that the
composition series $(5)$ is also a composition series of the
$\frak R(\mathrm{\bf I}; A)$-module $\dfrac{\frak R(\mathbf{I};
{I}_cM)}{\frak U^n \frak R(\mathbf{I}; {I}_cM)}.$
  Consequently $\ell_{\frak R(\mathrm{\bf
I}; A)}\bigg(\dfrac{\frak R(\mathbf{I}; {I}_cM)}{\frak U^n \frak
R(\mathbf{I}; {I}_cM)}\bigg) = u.$ From this it follows that
$$\ell_{\frak R(\mathrm{\bf
F}; A)}\bigg(\dfrac{{I}_c\frak R(\mathbf{F}; M) }{\frak
V^n{I}_c\frak R(\mathbf{F}; M)}\bigg) = \ell_{\frak R(\mathrm{\bf
I}; A)}\bigg(\dfrac{\frak R(\mathbf{I}; {I}_cM)}{\frak U^n \frak
R(\mathbf{I}; {I}_cM)}\bigg)$$ for all $n$. Therefore
$$e(\frak V; {I}_c\frak R(\mathbf{F}; M)) = e(\frak U; \frak
R(\mathbf{I}; {I}_cM)).\eqno(6)$$ By the assumption
$\mathrm{ht}\dfrac{I + \mathrm{Ann}M}{\mathrm{Ann}M}> 0,$ it
implies that $\mathrm{ht}\dfrac{{I}_c +
\mathrm{Ann}M}{\mathrm{Ann}M}> 0.$ So
 $$\dim M/{I}_cM < \dim M .$$  Hence from the short
exact sequence of $A$-modules: $0\rightarrow {I}_cM \rightarrow M
\rightarrow M/{I}_cM\rightarrow 0,$ we get
 $$e(\frak U; \frak R(\mathbf{I}; {I}_cM)) = e(\frak U; \frak R(\mathbf{I};
 M))\eqno(7)$$  by \cite[Corollary 3.9 (ii)(b)]{VT1} on the additivity
 of
 multiplicities of Rees modules.
Since $\mathrm{ht}\dfrac{I_c + \mathrm{Ann}M}{\mathrm{Ann}M}> 0,$
we have $ \dim \frak R(\mathbf{F}; M)/{I}_c\frak R(\mathbf{F}; M)<
\dim \frak R(\mathbf{F}; M).$ Therefore from the short exact
sequence of $\frak R(\mathrm{\bf F}; A)$-modules:
 $$0\rightarrow {I}_c\frak R(\mathbf{F}; M) \rightarrow \frak R(\mathbf{F}; M) \rightarrow
 \frak R(\mathbf{F}; M)/{I}_c\frak R(\mathbf{F}; M)\rightarrow 0,$$
 it follows that $$e(\frak V; {I}_c\frak R(\mathbf{F}; M)) = e(\frak V; \frak R(\mathbf{F};
 M)) \eqno(8)$$ (see e.g. \cite[Theorem
11.2.3]{HS}).
 Consequently, by (6), (7) and (8),  we obtain
 $$e(\frak V;\frak R(\mathbf{F}; M)) = e(\frak U; \frak R(\mathbf{I}; M)).$$
\end{proof}
  A version of Theorem \ref{thm2.3} (i) for the case of reductions of ideals in Noetherian local rings
   was proved by Viet in \cite[Theorem 4.1]{Vi3} by a different approach.
   Theorem \ref{thm2.3} is an important key which help us to obtain the following
   results for mixed multiplicities and the multiplicity  of Rees modules
   by short arguments.

Let $I$ be an ideal of $A$ and $t$ a variable over $A,$ and $\frak
R(I; A) = A[It]$ the Rees algebra of $I.$ Denote by
$\overline{\frak R(I; A)}$ the integral closure of $A[It]$ in the
polynomial ring $A[t].$ Then $\overline{\frak R(I;
A)}=\bigoplus_{n \ge 0}\overline{I^n}t^n$ (see \cite[Proposition
5.2.1]{HS}). Moreover, if $A$ is an analytically unramified ring,
then $\{\overline{I^n}\}_{n\geqslant 0}$ is an $I$-good filtration
(see \cite{R2}). In this case, $I$ is a reduction of
$\{\overline{I^n}\}_{n\geqslant 0}.$ Hence if we assume further
that $\mathrm{ht}I>0$, then $e(\overline{\frak R(I; A)}) =e(\frak
R(I;A))$ by Theorem \ref{thm2.3} (ii). We get the following
result.
\begin{corollary}\label{co2.7} Let $A$ be an analytically unramified ring and $I$ an
ideal of positive height of $A.$  Then $e(\overline{\frak R(I;
A)}) =e(\frak R(I; A)).$
\end{corollary}

Let $\frak I_1$ be a sequence consisting $k_1$ elements of
$(F_1)_1,\ldots,$ $\frak I_s$ be a sequence consisting $k_s$
elements of $(F_s)_1$ with $k_1,\ldots,k_s \ge 0.$
 Put   $(\emptyset) = 0_A$ and $\frak R = \frak I_1,\ldots,\frak I_s.$
 For any $1\le i\le s,$ set
$\varepsilon_i = (1, \ldots, 1, 0, 1, \ldots, 1) \in \mathbb{N}^s$
(the $i$th coordinate is 0). Then $\frak R$ is called a {\it joint
reduction} of filtrations $\mathbf{F}$ with respect to $M$ of the
type $\mathbf k = (k_1,\ldots,k_s)$ if $\mathbb{F}_{\mathbf{n} +
\mathbf{1}}M = \sum_{i=1}^s(\frak I_i) \mathbb{F}_{\mathbf{n} +
\mathbf{\varepsilon}_i}M$ for all large $\bf n.$ Recall that the
concept of joint reductions of $\mathfrak{m}$-primary ideals  was
given by Rees in 1984 \cite{Re}. This concept was extended to the
set of arbitrary ideals by \cite{Oc, Vi2,  Vi3,  VDT, {VT4}}.

We obtain the following corollary on expressing  mixed
multiplicities of modules with respect to filtrations in terms of
the multiplicity of their joint reductions, which is an extension
of \cite[Theorem 3.1]{VDT} and \cite[Theorem 2.4]{Re}.

\begin{corollary}\label{co2.8}
 Let $ F $ be an $\mathfrak{m}$-primary good filtration and  $\mathbf{F}$ be good filtrations in $A$.
 Set  $\dim M = d$ and $I= (F_1)_1 \cdots (F_s)_1.$ Assume  that $\mathrm{ht}
 \Big(\dfrac{I+\mathrm{Ann}_AM}{\mathrm{Ann}_AM}\Big) = h >0$ and $k_0 \in \mathbb{N}$,
${\bf k} = (k_1,\ldots,k_s) \in \mathbb{N}^s$
  such that $k_0+ |\mathbf{k}| = d-1$ and $|\mathbf{k}| < h.$
Let $\frak R = \frak I_0,\frak I_1,\ldots,\frak I_s$ be a joint
reduction of
 $F, \mathbf{F}$ with respect to $M$ of the type $(k_0+1, {\bf k})$ such that
$\frak R$ is a system of parameters for $M.$ Then
$$e_A(F^{[k_0+1]}, \mathbf{F}^{[\mathbf{k}]}; M) = e_A(\frak R;
M).$$
\end{corollary}
\begin{proof}
For each $i = 1, \ldots, s$, set $I_i = (F_i)_1$ and $I_0 =
(F)_1$. Since $I_0, I_1, \ldots, I_s$ are reductions of $F,
\mathbf{F}$, there exists a large enough positive integer $c$ such
that
$$(F)_{n_0+c}\mathbb{F}_{\mathbf{n} + c\mathbf{1}} = I_0^{n_0}\mathbb{I}^{\mathbf{n}}\mathcal{I}_c$$
for all $n_0, \mathbf{n}$, here $\mathcal{I}_c =
(F)_c(F_1)_c\cdots (F_s)_c$ (see Note \ref{no2.1}).
 Then by the assumption
that $\frak R$ is a joint reduction of
 $F, \mathbf{F}$ with respect to $M$, we have
 \begin{align*}
 I_0^{n_0+1}\mathbb{I}^{\mathbf{n} + \mathbf{1}}\mathcal{I}_cM &= (F)_{n_0+c+1}\mathbb{F}_{\mathbf{n} + (c+1)
 \mathbf{1}}M \\
 & =  \sum_{i=0}^s(\frak I_i)
(F)_{n_0+c+u_i}\mathbb{F}_{\mathbf{n} + c\mathbf{1} +
\varepsilon_i}M = \sum_{i=0}^s(\frak I_i)
I_0^{n_0+u_i}\mathbb{I}^{\mathbf{n} +
\mathbf{\varepsilon}_i}\mathcal{I}_cM
\end{align*}
for all large $n_0, \bf n$, here $\varepsilon_0 = \mathbf 1 \in
\mathbb{N}^s,$ $\varepsilon_i = (1, \ldots, 1, 0, 1, \ldots, 1)
\in \mathbb{N}^s$ (the $i$th coordinate is 0) and $u_0 = 0,$ $u_i
= 1$ for all $1 \le i \le s.$ Thus
$$I_0^{n_0+1}\mathbb{I}^{\mathbf{n} + \mathbf{1}}\mathcal{I}_cM =
\sum_{i=0}^s(\frak I_i) I_0^{n_0+u_i}\mathbb{I}^{\mathbf{n} +
\mathbf{\varepsilon}_i}\mathcal{I}_cM$$ for all large $n_0, \bf
n.$ Hence $\frak R$ is a joint reduction of
 $I_0, \mathbf{I}$ with respect to $\mathcal{I}_cM$. Recall that
$I= (F_1)_1 \cdots (F_s)_1,$  $\mathcal{I}_c = (F)_c(F_1)_c\cdots
(F_s)_c,$  $(F)_c$ is $\mathfrak{m}$-primary, and
  $\mathrm{ht}\dfrac{I + \mathrm{Ann}M}{\mathrm{Ann}M}=h> 0.$
  Hence $\dim M = \dim \mathcal{I}_cM$ and $\mathrm{ht}\dfrac{I + \mathrm{Ann}\mathcal{I}_cM}{\mathrm{Ann}\mathcal{I}_cM} =h.$
 So by \cite[Theorem 3.1]{VDT}, we get
 $e(I_0^{[k_0 + 1]},\mathbf{I}^{[\mathbf{k}]}; \mathcal{I}_cM) = e_A(\frak R; \mathcal{I}_cM).$
Moreover, by $(1)$ in the proof of Theorem \ref{thm2.3} (i), we
have
 $e(F^{[k_0 + 1]},\mathbf{F}^{[\mathbf{k}]}; M) = e(I_0^{[k_0 + 1]},\mathbf{I}^{[\mathbf{k}]};  \mathcal{I}_cM).$
 Therefore
 $$e(F^{[k_0 + 1]},\mathbf{F}^{[\mathbf{k}]}; M) = e_A(\frak R; \mathcal{I}_cM).$$
 Since $\mathrm{ht}\dfrac{I + \mathrm{Ann}M}{\mathrm{Ann}M}>
0$, $\mathrm{ht}\dfrac{\mathcal{I}_c +
\mathrm{Ann}M}{\mathrm{Ann}M}> 0.$ So $\dim M/\mathcal{I}_cM <
\dim M.$ Hence from the short exact sequence of $A$-modules:
$0\rightarrow \mathcal{I}_cM \rightarrow M \rightarrow
M/\mathcal{I}_cM\rightarrow 0,$ it follows that $e_A(\frak R;
\mathcal{I}_cM) = e_A(\frak R; M)$ (see e.g. \cite[Theorem
11.2.3]{HS}).  Consequently we obtain
 $e(F^{[k_0 + 1]},\mathbf{F}^{[\mathbf{k}]}; M) = e_A(\frak R; M).$
\end{proof}

\vspace{12pt}
\section{On some formulas for  multiplicities
 } \vskip 0.2cm
 \noindent
Continuing to apply the main theorem (Theorem \ref{thm2.3}), in
this section, we give some formulas for transforming  mixed
multiplicities and multiplicities of Rees modules of good
filtrations.

Keep the notations in Theorem \ref{thm2.3}. Let $J$ be a reduction
of $F$. We choose the reductions $I_1=(F_1)_1, \ldots,
I_s=(F_s)_1$ of filtrations $F_1, \ldots, F_s$, respectively.  Set
$\dim M = d$ and $I= (F_1)_1 \cdots (F_s)_1.$
  If $\mathrm{ht}\dfrac{I +
\mathrm{Ann}M}{\mathrm{Ann}M}> 0$, then  by \cite[Theorem
4.4]{HHRT} which is a generalized version of \cite[Theorem
1.4]{Ve} (see \cite{VT1}), we have
$$
e\big(\big(J,\mathfrak R(\mathrm{\bf I}; A)_+\big); \mathfrak
R(\mathrm{\bf I}; M)\big)= \sum_{k_0\:+\:\mid\mathrm{\bf k}\mid
=\;d-1}e\big(J^{[k_0+1]},\mathrm{\bf I}^{[\mathrm{\bf
k}]};M\big).$$ On one hand, by Theorem \ref{thm2.3} (ii), we get
$$e\big(\big(J,\mathfrak R(\mathrm{\bf F}; A)_+\big); \mathfrak R(\mathrm{\bf F}; M)\big)
=e\big(\big(J,\mathfrak R(\mathrm{\bf I}; A)_+\big); \mathfrak
R(\mathrm{\bf I}; M)\big).$$ On the other hand, by Theorem
\ref{thm2.3} (i), it follows that
$$\sum_{k_0\:+\:|\mathrm{\bf k}|
=\;d-1}e\big(F^{[k_0+1]},\mathrm{\bf F}^{[\mathrm{\bf k}]};M\big)=
\sum_{k_0\:+\:\mid\mathrm{\bf k}\mid
=\;d-1}e\big(J^{[k_0+1]},\mathrm{\bf I}^{[\mathrm{\bf
k}]};M\big).$$ Hence we obtain the following result.

\begin{corollary}\label{co2.4} Let $ F$ be an $\mathfrak{m}$-primary good filtration and  $\mathbf{F}$ be good
filtrations. Let $J$ be a reduction of $F$. Set $\dim M = d$  and
$I= (F_1)_1 \cdots (F_s)_1.$ Assume that $\mathrm{ht}\dfrac{I +
\mathrm{Ann}M}{\mathrm{Ann}M}> 0$. Then
$$e\big(\big(J,\mathfrak R(\mathrm{\bf F}; A)_+\big);
\mathfrak R(\mathrm{\bf F}; M)\big)= \sum_{k_0\:+\:|\mathrm{\bf
k}| =\;d-1}e\big(F^{[k_0+1]},\mathrm{\bf F}^{[\mathrm{\bf
k}]};M\big).$$
\end{corollary}

 The next corollary is the  additivity on exact sequences of mixed multiplicities and
the multiplicity of Rees modules of filtrations. Let $ W_1, W_2,
W_3$ be finitely generated  $A$-modules and $ 0\longrightarrow W_1
\longrightarrow W_3 \longrightarrow W_2\longrightarrow 0 $ be a
short exact sequence of $A$-modules. For any $i = 1, 2, 3,$ set
$\overline{W}_i= \dfrac{W_i}{0_{W_i}: I^\infty}$. Suppose that $I=
(F_1)_1 \cdots (F_s)_1$ is not contained in
$\sqrt{\mathrm{Ann}{W_i}}$ \;for all $i = 1, 2, 3.$

From Theorem \ref{thm2.3} and \cite[Corollary 3.9]{VT1} and
Corollary \ref{co2.4}, we get the following result on the
additivity on exact sequences of multiplicities of filtrations.
\begin{corollary}\label{co2.5} Keep the above notations.
Let $\frak I$ be an $\frak m$-primary ideal of $A$. Set $\frak J =
(\frak I, \mathfrak R(\mathrm{\bf F}; A)_+)$.
    Assume that $k_0 \in \mathbb{N}$,
${\bf k} \in \mathbb{N}^s$
  such that $k_0+ 1+ |\mathbf{k}| = \dim \overline{W}_3.$  Then the following statements hold.
 \begin{itemize}
 \item[$\mathrm{(i)}$] If $\dim  \overline{W}_1 =\dim \overline{W}_2=\dim \overline{W}_3$,
 then
\begin{align*}
(a )&: e(F^{[k_0+1]},\mathrm{\bf F}^{[\mathrm{\bf k}]}; W_3)=
e(F^{[k_0+1]},\mathrm{\bf F}^{[\mathrm{\bf k}]}; W_1)+
e(F^{[k_0+1]},\mathrm{\bf F}^{[\mathrm{\bf k}]}; W_2).\\
(b)&: \text{If } \mathrm{ht}\dfrac{I +
\mathrm{Ann}W_i}{\mathrm{Ann}W_i} >0
\text{ for all } i = 1, 2, 3, \text{ then}\\
&\quad\quad e\big(\frak J; \mathfrak R(\mathrm{\bf F}; {W}_3)\big)
= e\big(\frak J; \mathfrak R(\mathrm{\bf F}; {W}_1)\big)+
e\big(\frak J; \mathfrak R(\mathrm{\bf F}; {W}_2)\big).
\end{align*}

\item[$\mathrm{(ii)}$]  If $h\ne k = 1,2$ and $\dim \overline{W}_3
> \dim \overline{W}_h$, then
\begin{align*}
(a)&: e(F^{[k_0+1]},\mathrm{\bf F}^{[\mathrm{\bf k}]}; W_3)=
e(F^{[k_0+1]},\mathrm{\bf F}^{[\mathrm{\bf k}]}; W_k).\;\;\;\;\;\;\;\;\;\;\;\;\;\;\;\;\;\;\;\;\;\;\;\;\;\;\;\;\;\;\;\\
(b)&: \text{If } \mathrm{ht}\dfrac{I + \mathrm{Ann}W_i}{\mathrm{Ann}W_i} >0 \text{ for all } i = 1, 2, 3, \text{ then}\\
&\quad\quad \quad\quad e\big(\frak J; \mathfrak R(\mathrm{\bf F};
{W}_3)\big) = e\big(\frak J; \mathfrak R(\mathrm{\bf F};
{W}_k)\big).
\end{align*}
\end{itemize}
\end{corollary}
\begin{proof} The proof of (i)(a) and (ii)(a):
We choose the reductions $$J= (F)_1, I_1=(F_1)_1, \ldots,
I_s=(F_s)_1$$ of filtrations $F, F_1, \ldots, F_s$, respectively.
Then by Theorem \ref{thm2.3} (i), we obtain
$$e(F^{[k_0+1]},\mathrm{\bf F}^{[\mathrm{\bf k}]}; W_i)= e(J^{[k_0+1]},\mathrm{\bf I}^{[\mathrm{\bf k}]};
W_i)$$ for all $1 \le i \le 3.$ Consequently, the part (i)(a) and
the part (ii)(a) follow  immediately  from \cite[Corollary 3.9
(i)(a) and (ii)(a)]{VT1}, respectively.

 The proof of (i)(b): Note that since $\mathrm{ht}\dfrac{I + \mathrm{Ann}W_i}{\mathrm{Ann}W_i}
 >0,$  $\dim  \overline{W}_i= \dim {W}_i.$ Hence $\dim  {W}_1 =\dim {W}_2=\dim {W}_3.$
 Set $\dim W_i = d$ for $i = 1, 2, 3$ and
 $\mathcal F= \{\frak I^n\}_{n\geqslant 0}$ the $\frak I$-adic filtration. On one hand, by Corollary
 \ref{co2.4}, we have
$$e\big(\frak J; \mathfrak R(\mathrm{\bf F}; {W}_i)\big) = \sum_{k_0\:+\:|\mathrm{\bf k}|
=\;d-1}e\big(\mathcal F^{[k_0+1]},\mathrm{\bf F}^{[\mathrm{\bf
k}]}; {W}_i\big) \eqno(9)$$ for $i = 1, 2, 3$.  On the other hand
by (i)(a), it follows that
\begin{align*}
\sum_{k_0\:+\:|\mathrm{\bf k}| =\;d-1}e\big(\mathcal
F^{[k_0+1]},\mathrm{\bf F}^{[\mathrm{\bf k}]}; {W}_3\big) =
&\sum_{k_0\:+\:|\mathrm{\bf k}| =\;d-1}e\big(\mathcal
F^{[k_0+1]},\mathrm{\bf F}^{[\mathrm{\bf k}]}; {W}_1\big) \\ &+
\sum_{k_0\:+\:|\mathrm{\bf k}| =\;d-1}e\big(\mathcal
F^{[k_0+1]},\mathrm{\bf F}^{[\mathrm{\bf k}]}; {W}_2\big).
\end{align*}
Hence by (9), we obtain $e\big(\frak J; \mathfrak R(\mathrm{\bf
F}; {W}_3)\big) = e\big(\frak J; \mathfrak R(\mathrm{\bf F};
{W}_1)\big)+e\big(\frak J; \mathfrak R(\mathrm{\bf F};
{W}_2)\big).$ The proof of (ii)(b): Similarly to the proof of
(i)(b), by (9) and (ii)(a), we  get (ii)(b).
\end{proof}
We also easily prove the following  additivity and reduction
formulas for mixed multiplicities and the multiplicity  of Rees
modules of filtrations which  are generalizations  of
\cite[Theorem 3.2]{VT1} and \cite[Theorem 17.4.8]{HS}.
\begin{corollary}\label{co2.6}
Let $ F $ be an $\mathfrak{m}$-primary good filtration and  $
\mathbf{F}$ be good filtrations  in $A$ such that $I= (F_1)_1
\cdots (F_s)_1$ is not contained in $\sqrt{\mathrm{Ann}M}$.  Set
$\overline{M}=\dfrac{M}{0_M: I^\infty}.$ Denote by  $\Pi$ the set
of all prime ideals $\frak p $ of  $A$ such that $\frak p \in
\mathrm{Min}(A/\mathrm{Ann}\overline{M})$ and $\dim A/\frak p =
\dim \overline{M}.$ Let $\frak I$ be an $\frak m$-primary ideal of
$A$. Then we have
\begin{enumerate}[{\rm (i)}]
\item $e(F^{[k_0+1]},\mathrm{\bf F}^{[\mathrm{\bf k}]};M)=
\sum_{\frak p \in \Pi}\ell({M}_{\frak p})e(F^{[k_0+1]},\mathrm{\bf
F}^{[\mathrm{\bf k}]};A/\frak p).$ \item     If
$\mathrm{ht}\dfrac{I + \mathrm{Ann}M}{\mathrm{Ann}M}> 0$, then
$$e\big((\frak I,\mathfrak R(\mathrm{\bf F}; A)_+) ;\mathfrak R(\mathrm{\bf F}; M) \big)
= \sum_{\frak p \in \Pi}\ell({M}_{\frak p})e\big((\frak
I,\mathfrak R(\mathrm{\bf F}; A)_+);
 \mathfrak R(\mathrm{\bf F}; A/\frak p)\big).$$
\end{enumerate}
\end{corollary}
\noindent {\it Proof.} The proof of (i): Choose the reductions $J=
(F)_1, I_1=(F_1)_1, \ldots, I_s=(F_s)_1$ of filtrations $F, F_1,
\ldots, F_s$, respectively. Then on one hand by Theorem
\ref{thm2.3} (i), we get $e(F^{[k_0+1]},\mathrm{\bf
F}^{[\mathrm{\bf k}]};M)=e(J^{[k_0+1]},\mathrm{\bf
I}^{[\mathrm{\bf k}]};M)$ and $e(F^{[k_0+1]},\mathrm{\bf
F}^{[\mathrm{\bf k}]};A/\frak p)= e(J^{[k_0+1]},\mathrm{\bf
I}^{[\mathrm{\bf k}]};A/\frak p)$ for all $\frak p \in \Pi.$ On
the other hand by \cite[Theorem 3.2]{VT1}, we have
$$e(J^{[k_0+1]},\mathrm{\bf I}^{[\mathrm{\bf k}]};M)= \sum_{\frak p
\in \Pi}\ell({M}_{\frak p})e(J^{[k_0+1]},\mathrm{\bf
I}^{[\mathrm{\bf k}]};A/\frak p).$$  So
$e(F^{[k_0+1]},\mathrm{\bf F}^{[\mathrm{\bf k}]};M)= \sum_{\frak p
\in \Pi}\ell({M}_{\frak p})e(F^{[k_0+1]},\mathrm{\bf
F}^{[\mathrm{\bf k}]};A/\frak p).$

The proof of (ii): Set $\dim M
=d$ and  $\mathcal F= \{\frak I^n\}_{n\geqslant 0}$ the $\frak
I$-adic filtration. By Corollary \ref{co2.4}  and (i) we have
\begin{align*}
 e\big((\frak I,\mathfrak R(\mathrm{\bf F}; A)_+); \mathfrak R(\mathrm{\bf F}; M)\big) &
 = \sum_{k_0\:+\:|\mathrm{\bf k}|
=\;d-1}e\big(\mathcal F^{[k_0+1]},\mathrm{\bf F}^{[\mathrm{\bf k}]}; M\big)\\
& = \sum_{k_0\:+\:|\mathrm{\bf k}| =\;d-1}\Big(\sum_{\frak p \in
\Pi}\ell({M}_{\frak p})e(\mathcal F^{[k_0+1]},
\mathrm{\bf F}^{[\mathrm{\bf k}]};A/\frak p)\Big)\\
& =\sum_{\frak p \in \Pi}\ell({M}_{\frak p})
\Big(\sum_{k_0\:+\:|\mathrm{\bf k}|
=\;d-1}e(\mathcal F^{[k_0+1]},\mathrm{\bf F}^{[\mathrm{\bf k}]};A/\frak p)\Big)\\
& =\sum_{\frak p \in \Pi}\ell({M}_{\frak p}) e\big((\frak
I,\mathfrak R(\mathrm{\bf F}; A)_+); \mathfrak R(\mathrm{\bf F};
A/\frak p)\big). \quad \square
\end{align*}

Denote by $U$ the set of all non-zero divisors of $A.$ Recall that
an $A$-module $N$ has {\it rank} $r$ if $U^{-1}N$ (the
localization of $N$ with respect to $U$) is a free
$U^{-1}A$-module of rank $r.$

Now, using Theorem \ref{thm2.3} (i) and \cite [Theorem 3.4]{VT2}
we  give the following.

\begin{corollary}\label{co3.b}  Let $ F$ be an $\mathfrak{m}$-primary good filtration
and $\mathbf{F}$ be good filtrations of ideals in $A$ such that
$I= (F_1)_1 \cdots (F_s)_1$ is not contained in
$\sqrt{\mathrm{Ann}M}$. Suppose that $M$ has rank $r>0$. Then
 $e\big({F}^{[k_0+1]},{\mathrm{\bf F}}^{[\mathrm{\bf k}]}; M\big)
= e\big(F^{[k_0+1]},\mathrm{\bf F}^{[\mathrm{\bf k}]};
A\big)\mathrm{rank}_AM.$
\end{corollary}
\begin{proof} Choose the reductions $J= (F)_1,
I_1=(F_1)_1, \ldots, I_s=(F_s)_1$ of filtrations $F, F_1, \ldots,
F_s$, respectively. Then by Theorem \ref{thm2.3} (i), we have
$$e(F^{[k_0+1]},\mathrm{\bf F}^{[\mathrm{\bf
k}]};A)=e(J^{[k_0+1]},\mathrm{\bf I}^{[\mathrm{\bf k}]};A)$$ and
$e(F^{[k_0+1]},\mathrm{\bf F}^{[\mathrm{\bf
k}]};M)=e(J^{[k_0+1]},\mathrm{\bf I}^{[\mathrm{\bf k}]};M).$ By
\cite [Theorem 3.4]{VT2}, we obtain
$$e\big({J}^{[k_0+1]},{\mathrm{\bf I}}^{[\mathrm{\bf k}]}; M\big)
= e\big(J^{[k_0+1]},\mathrm{\bf I}^{[\mathrm{\bf k}]};
A\big)\mathrm{rank}_AM.$$ So $e\big({F}^{[k_0+1]},{\mathrm{\bf
F}}^{[\mathrm{\bf k}]}; M\big) = e\big(F^{[k_0+1]},\mathrm{\bf
F}^{[\mathrm{\bf k}]}; A\big)\mathrm{rank}_AM.$
\end{proof}

\noindent {\bf Acknowledgement:} {\small This research is funded by Vietnam National Foundation for Science and Technology Development (NAFOSTED) under grant number 101.04.2015.01.}


{\small
\begin{thebibliography}{99}

\bibitem{Bh}P. B. Bhattacharya, {\it The Hilbert function of two ideals}, Proc. Cambridge Philos. Soc. {53} (1957), 568-575.



\bibitem{DV} L. V. Dinh and D. Q. Viet, {\it On two results of mixed multiplicities}, Int. J. Algebra 4(1) 2010, 19-23.
\bibitem{HHRT} M. Herrmann, E.  Hyry,  J.  Ribbe, Z. Tang,  {\it Reduction numbers and multiplicities of multigraded
structures},  J. Algebra 197(1997), 311-341.
\bibitem{HM}S. Huckaba and T. Marley, {\it Hilbert coefficients and the depths of associated graded rings},
J. London Math. Soc, {56} (1997), 64-76.
\bibitem{HS} C. Huneke and I. Swanson, {\it Integral Closure of Ideals, Rings, and Modules}, London Mathematical
Lecture Note Series 336, Cambridge University Press (2006).
\bibitem{KR1} D. Kirby and D. Rees, {\it Multiplicities in graded rings I: the general theory}, Contemporary
Mathematics 159(1994), 209-267.
\bibitem{KR2} D. Kirby and D. Rees, {\it Multiplicities in graded rings II: integral equivalence and the
Buchsbaum -Rim multiplicity}, Math. Proc. Cambridge Phil. Soc. 119
(1996), 425-445.
\bibitem{KV} D. Katz and J. K. Verma, {\it Extended Rees algebras and mixed multiplicities}, Math. Z. 202(1989), 111-128.
\bibitem{MV} N. T. Manh and D. Q. Viet, {\it Mixed  multiplicities of modules over Noetherian local rings}, Tokyo J. Math.
 29(2006), 325-345.
  \bibitem{RR}L. J. Ratliff and D. Rush, {\it Two notes on reductions of ideals}, Indiana Univ. Math. J. {27} (1978), 929-934.
  \bibitem{Oc} L. O'Carroll, {\it On two theorems concerning reductions in local rings}, J. Math. Kyoto Univ. 27-1(1987), 61-67.
\bibitem{R2}D. Rees, {\it A note on analytically unramified local rings}, J. London Math. Soc. {36} (1961), 24-28.
\bibitem{Re} D. Rees, {\it Generalizations of reductions and mixed multiplicities}, J. London. Math. Soc. 29(1984),
 397-414.
\bibitem{Sw} I. Swanson, {\it Mixed multiplicities, joint reductions and quasi-unmixed local rings},
J. London Math.Soc. 48(1993), no.1, 1-14.
\bibitem{Te} B. Teissier, {\it Cycles \`evanescents, sections planes, et conditions de Whitney},
Singularities \`a Carg\'ese, 1972. Ast\'erisque, 7-8(1973),
285-362.

\bibitem{Ve} J. K. Verma, {\it Multigraded  Rees algebras and mixed multiplicities}, J. Pure and  Appl.
 Algebra 77(1992), 219-228.

\bibitem{Vi1} D. Q. Viet, {\it Mixed multiplicities of arbitrary ideals in local rings}, Comm. Algebra. 28(2000), 3803-3821.
\bibitem{ViFC} D. Q. Viet, {\it On some properties of $(FC)$-sequences of ideals in local rings}, Proc. Amer. Math. Soc. 131 (2003), 45-53.
\bibitem{Vi2} D. Q. Viet, {\it Sequences determining mixed multiplicities and reductions of ideals}, Comm. Algebra. 31(2003), 5047-5069.
\bibitem{Vi3} D. Q. Viet, {\it Reductions and mixed multiplicities of ideals}, Comm. Algebra. 32(2004), 4159-4178.
\bibitem{Vi5} D. Q. Viet, {\it The multiplicity and the Cohen-Macaulayness of extended Rees
algebras of equimultiple ideals}, J. Pure and Appl. Algebra 205
(2006), 498-509.
\bibitem{DQV} D. Q. Viet, {\it On the Cohen-Macaulayness of fiber cones},  Proc. Amer. Math. Soc. 136 (2008),  4185-4195.

\bibitem{VD1} D. Q. Viet and  L. V. Dinh, {\it On the multiplicity of Rees algebras of good filtrations},
 Kyushu J. Math. 66 (2012), 261-272.

\bibitem{VD} D. Q. Viet and L. V. Dinh, {\it On Mixed Multiplicities of good filtrations}, Algebra Colloq. 22, 421 (2015) 421-436.


\bibitem{VDT} D. Q. Viet and L. V. Dinh, T. T. H. Thanh,
{\it A note on joint reductions and mixed multiplicities}, Proc.
Amer. Math. Soc. 142 (2014), 1861-1873.
\bibitem{VM} D. Q. Viet and N. T. Manh,  {\it Mixed multiplicities of multigraded modules},
Forum Math. 25 (2013), 337-361.
\bibitem{VT} D. Q. Viet and T. T. H. Thanh, {\it On $(FC)$-sequences and mixed multiplicities of multi-graded algebras}, Tokyo J. Math. 34 (2011), 185-202.
\bibitem{VT3} D. Q. Viet and T. T. H. Thanh, {\it Multiplicity and Cohen-Macaulayness of fiber cones
of good filtrations}, Kyushu J. Math. 65(2011), 1-13.
\bibitem{VT1} D. Q. Viet and T. T. H. Thanh, {\it On  some  multiplicity and mixed multiplicity formulas}, Forum Math. 26(2014), 413-442.
\bibitem{VT2} D. Q. Viet and  T. T. H. Thanh, {\it A note on formulas transmuting
mixed multiplicities}, Forum Math. 26 (2014), 1837-1851.
\bibitem{VT4} D. Q. Viet and  T. T. H. Thanh, {\it On filter-regular sequences of
multi-graded modules}, Tokyo J. Math. 38(2015), 439-457.


 \end{thebibliography}}
\end{document}